\documentclass[12pt,reqno]{article}
\usepackage{amssymb, amsmath}
\usepackage{amssymb,amsthm}
\newtheorem{Theorem}{Theorem}
\newtheorem{Remark}{Remark}
\newtheorem{Lemma}{Lemma}
\newtheorem{Proposition}{Proposition}
\newtheorem{Corollary}{Corollary}

\numberwithin{Definition}{section}
\numberwithin{Theorem}{section}
\numberwithin{Lemma}{section}
\numberwithin{Proposition}{section}
\newtheorem{Example}{Example}
\numberwithin{Corollary}{section}
\numberwithin{Example}{section}
\numberwithin{Remark}{section}
\begin{document}
\title{Differentiability of Distance Function and The  Proximinal Condition implying Convexity}
\author{Triloki Nath\thanks{The author is supported by  UGC, Govt. of India, New Delhi-110012, under UGC-BSR Start-Up Grant No. F.30.12/2014(BSR), dated $22^{nd}$ July, 2014.} }
\maketitle
\begin{abstract}
  We establish a necessary and sufficient condition for the differentiability of the distance function generated by a nonempty closed set $K$ in a real  normed linear space $X$  under a proximinality condition on $K$. We do not assume the uniform differentiability constraints on the norm of the space as in Giles \cite{Giles1}. Hence,  our result advances that of Giles \cite{Giles1}. We prove that the proximinal condition of Giles \cite{Giles1} is true for almost suns.  The proximinal condition ensures  convexity of an almost sun in some class of strongly smooth spaces under a differentiability condition of the distance function. A necessary and sufficient condition is obtained for the convexity of Chebyshev sets in  Banach spaces with  rotund dual.

\end{abstract}
\section{Introduction}\label{sec:intro}
Let $X$ be a real normed linear space and $X^{\ast}$ be its dual space. For a nonempty closed set $K$ in $X$, we define its distance function $d_K$ on $X$ by $$d_K (x)= \inf \left\{\left\|x-k\right\| : k\in K\right\}.$$ This function is not necessarily everywhere differentiable but it is (globally) Lipschitz, with the Lipschitz constant equal to 1. The metric projection of $x$ onto $K$ is $$P_K (x)= \left\{k \in K : \left\|x-k\right\| = d_K(x)\right\}.$$ The set $K$ is called \emph{proximinal (Chebyshev)} if for every  $x \in  X \diagdown K, P_K (x) $ is nonempty (singleton). $K$ will be called \emph{almost proximinal} if $P_K (x) $ is nonempty for a dense set of $ x \in X \diagdown K.$
The important concept of a \emph{sun} in abstract approximation theory was first introduced by Efimov and Stechkin \cite{Eff-Stech}. For a nonempty set $K$ in a normed linear space $X$, a point $x \in X \diagdown K$ is called a \emph{solar point} of $K$ if there exists a point $p(x)\in P_K (x)$ such that the points  $x + t (x-p(x))$  have $p(x)$  as a nearest point for all $t \geqslant 0$.  A proximinal set $K$ in a normed linear space $X$ is a \emph{sun}  if  $ X \diagdown K$ is the set of solar points.  An almost proximinal set $K$ will be called an \emph{almost sun} if a dense set of $X \diagdown K$ is the set of solar points.

Dutta \cite{Dutt1} has deduced that if the norm on $X$ is \emph{locally uniformly
convex}  (LUR), defined in section \ref{sec:MResl}, and \emph{$($Fr\'echet$)$ smooth}, then the  (Fr\'echet) smoothness of the distance function $d_{K}$ generated by an almost proximinal set $K$ is generic on $X \diagdown K.$ Further,  if norms on $X$ and $X^{\ast}$ are LUR, then he characterized the convexity of Chebyshev sets in terms of the Clarke generalized subdifferential of the distance function. His  technique is based on the observation of the denseness of the set $T'(K)$, where  $T'(K)$ denotes the set of points in $X \diagdown K$ for which every minimizing sequence in $K$ converges to a unique nearest point. A sufficient condition for $T'(K)$ to be dense is the local uniform convexity of the norm on $X$. We will use this result to improve the results of Giles \cite{Giles1}.

 In a normed linear space $X$, Giles \cite{Giles1} assumed a proximinal condition  on a nonempty closed set
$K,$ which has the property that for some $r > 0$  there exists a set of points
$x_{\circ} \in X  \diagdown K$ which have nearest points $ p(x_{\circ}) \in K $ with  $d_K(x_{\circ}) > r $ such that the set of points $\displaystyle x_{\circ}- r \frac{x_{\circ}-p(x_{\circ})}{\| x_{\circ}-p(x_{\circ})\|}$ is dense in $ X  \diagdown K $. It has been shown that if the norm has sufficiently strong differentiability properties, then the distance function $ d_K $ generated by $K$ has similar differentiability properties and it follows that, in some spaces, $K$ is convex.\\

It is well known that in a smooth finite-dimensional normed linear space every Chebyshev set is convex and the metric projection is continuous on $X \diagdown K$ and this fact is used in one of the proofs of this convexity result. So it is natural to impose certain continuity conditions on the metric projection while investigating the convexity of Chebyshev sets in smooth infinite-dimensional spaces. A milestone result is due to Vlasov \cite{Vlas1}:  in a Banach space with rotund dual, Chebyshev sets with continuous metric projection are convex. Later numerous results were proved on the convexity of Chebyshev sets with conditions on the set of points of discontinuity of the metric projection, e.g. see \cite{Balag}. A  close look of Vlasov's proof shows that  the continuity of the metric projection has been used only to establish a differentiability condition of the distance function generated by the set. In terms of a differentiability condition on the distance function, Vlasov's Theorem can be stated as follows.\\
\begin{Proposition} { \rm( Borwein et al. \cite[Theorems 14-18]{Borwn})}.\\
\label{prop_open} In a Banach space $X$ with rotund dual  $X^{\ast},$  a nonempty closed set  $K$ is convex if its distance function $ d_K $ satisfies
$$\limsup_{\left\| y \right\| \rightarrow 0}\dfrac{d_K(x+y)-d_K(x)}{\left\| y \right\|} = 1 ~~{\rm for~ all}~~ x \in X \diagdown K.$$
In particular, this differentiability condition is satisfied if $ d_K $  is smooth
and $\left\| d'_K(x)  \right\| = 1 $ or if $d_K $ is Fr\'echet smooth for all $x \in X \diagdown K.$
\end{Proposition}
It is easily seen that in any normed linear space $X,$ if every point $x \in X \diagdown K$ is an interior point of an interval with end points $x_{\circ} \in X \diagdown K$ and a nearest point $p(x_{\circ}) \in K,$ then $  p(x_{\circ}) \in P_K(x)$  that is $ p(x) = p(x_{\circ})$ for some $ p(x) \in P_K(x)$ and the  differentiability condition of Proposition \ref{prop_open} will be satisfied. We will use Vlasov's Theorem in the form of Proposition \ref{prop_open} to establish convexity results.  Fitzpatrick \cite{Fitz1} observed a close connection between the continuity of the metric projection and differentiability of the distance function.\\

In section \ref{sec:DPC} of the present paper, differentiability properties of the distance function are discussed and some new results are obtained regarding the property. In the last section, a proximinal condition is assumed and it is shown that the  sufficient condition for the proximinal condition is  almost sun property. Consequently,  almost sun property becomes key to establish the convexity of Chebyshev sets and some related results.\\

To make the paper self-contained, we reproduce some definitions and known results given as follows.

A  function  $h: X \rightarrow \mathbb R$ is said to be  \textit{G\^ateaux differentiable} or \textit{smooth} at $x \in X$ if there exists a continuous linear functional  $h'(x)\in X^{\ast}$, called the G\^ateaux  derivative of $h$, such that for any $\epsilon > 0$ and $y \in X$ with $\left\|y\right\|=1$ there exists  a $\delta(\epsilon, x,y) > 0$ such that

\begin{equation}
\left|\frac{h(x+ty)- h (x)}{t} - h'(x)(y)\right|< \epsilon \rm ~ {when} ~0 < \left|t\right| < \delta.
\label{c}
\end{equation}

A function $h$ is said to be \textit{Fr\'echet smooth} at $x$ if  there exists  a $\delta(\epsilon, x) > 0$ such that the inequality (\ref{c}) holds for all $y \in X$ with $\left\|y\right\|=1.$

 A function $h$ is said to be  \textit{uniformly smooth} on a set $D$  if  there exists  a $\delta(\epsilon, y) > 0$ such that the inequality (\ref{c}) holds for all $x \in D, $ and is said to be \textit{uniformly Fr\'echet smooth} on a set $D$  if  there exists  a $\delta(\epsilon) > 0$ such that the inequality (\ref{c}) holds for all $x \in D $ and for all $y \in X$ with $\left\|y\right\|=1.$

 A space $X$ is said to be \textit{smooth} (\textit{Fr\'echet smooth}) at $x \neq 0$ if the  norm is smooth (Fr\'echet smooth) at $x \neq 0$. We say that $X$ has \textit{uniformly smooth} (\textit{uniformly Fr\'echet smooth})  norm if the norm is uniformly smooth (uniformly Fr\'echet smooth)  on the unit sphere $\left\{x \in X: \left\|x\right\| = 1\right\}.$

Let  $h: X \rightarrow \mathbb R$ be a locally Lipschitz function. The \textit{Clarke generalized directional derivative }of $h$ at a point $x$ and in the direction $y\in X$, denoted by $h^\circ(x; y)$, is given by:
$$h^\circ(x; y)=\limsup_{{z\rightarrow x},~ {t\downarrow 0}} \frac{h (z+ty)- h(z)}{t} $$

and the \textit{Clarke generalized subdifferential} of $h$ at $x$ is given by
$$\partial h(x)= \left\{f \in X^\ast:h^\circ(x; y)\geqslant  f (y),~ \forall y \in X\right\}.$$

A locally Lipschitz function $h: X \rightarrow \mathbb R$ is said to be \emph{strictly smooth} (or \emph{strictly differentiable}) at $x \in X$ if there exists a continuous linear functional  $h'(x)\in X^{\ast}$, called the strict derivative of $h$, such that
\begin{equation}
\lim_{{z\rightarrow x}, {t\downarrow 0}}\frac{h(z+ty)- h (z)}{t} = h'(x)(y)
 \label{str.diff.1}
\end{equation}
for all $y$ in $X$. Of course, here the strict derivative is of G\^ateaux type.\\

The following definition of upper semicontinuity for the set-valued mappings is classical.\\

Let $T$ be a set-valued mapping from a topological space
$Y$ into the dual $X^{\ast}$ of a normed linear space $X$. Then  $T$  is said to be $nw^{\ast}$-\emph{upper semicontinuous} on $Y$ if for each $w^*$-open subset $W$ of $X^{\ast}, \{x\in Y : T (x) \subset W\}$ is
open in $Y$.

\section{Differentiability of the Distance Function}\label{sec:DPC}
Following the notations of \cite{Balag}, we denote by $E(K)$, the set of all points in $X \diagdown K$ which have nearest points in $K$,  the set $T(K)$ is the set of all points in $X \diagdown K$ which has a unique nearest point in $K$ and $AC(K)$ is the set of points  $x \in X \diagdown K$  such that every minimizing sequence  in $K$ for $x$ has a subsequence which converges to a point in $K$ i.e. $AC(K)$ is the set of points of \emph{approximative compactness} of $K$. Recall that $T'(K)$ is  the set of $x \in E(K)$ such that every minimizing sequence  in $K$ for $x$ converges to a unique nearest point of $x$. Clearly, $T'(K)= T(K)\cap AC(K)$.\\
  \hspace*{0.5 cm} We denote by  $f_x$, a subgradient of the norm at $x \in X$, then the subdifferential  $\partial \left\|x\right\|$, the set of all subgradients of the norm at $x \in X$,  is given by
 \begin{eqnarray*}
\partial \left\|x\right\|= \left\{f_x  \in X^{\ast}: f_x (x) = \left\|x\right\|~ {\rm{and}}~ \left\|f_x\right\| = 1\right\}
\end{eqnarray*}
Note that $\partial \left\| \frac{x}{\left\|x\right\|} \right\|= \partial \left\|x\right\|$ for $x \neq 0$.  If the norm is smooth at $x \neq 0 $ then $\partial \left\|x\right\|$  is a singleton. In this case the single subgradient  $f_x$  becomes the G\^ateaux derivative.\\

The following two lemmas play very crucial role in establishing our main results.

\begin{Lemma} {\rm (Borwein and Giles \cite[Lemma 1]{Borw1})} For any $z \in E(K)$ and every $p(z) \in P _{K}(z)$, there exists an $f_{z- p(z)} \in \partial \left\|z- p(z)\right\|$ such that $f_{z- p(z)} \in \partial d_{K}(z).$
\label{lm0}
\end {Lemma}
\begin{Lemma} {\rm (Dutta \cite[Lemma 1]{Dutt1})}
 For any $z \in T'(K)$ we have $$\partial d_{K}(z) \subseteq  \partial \left\|z- P_{K}(z)\right\|.$$
\label{lm1}
The equality holds if the norm on $X$ is  smooth at $z- P_{K}(z), z \notin K$. Moreover, for $z \notin K$,  if the norm on $X$ is Fr\'echet smooth at $z- P_{K}(z)$, then $d_{K}$ is Fr\'echet smooth at $z.$
\end {Lemma}
For a more detailed explanation of the generalized subdifferential, see Clarke \cite{Clarke1}. We refer \cite{Dudov} for subdifferentiability and superdifferentiability (in the Dem'yanov-Rubinov sense) of the distance function. \\
\begin{Remark}
\label{rmk1}
Since $d_{K} $ is a Lipschitz function, the condition that $\partial d_{K}(x)$
 is singleton is equivalent to strict differentiability of $d_{K} $ at $x$, {\rm see \cite[Proposition 2.2.4]{Clarke1}}.


In  Lemma \ref{lm1}, if the norm is smooth $($Fr\'echet smooth$)$, then clearly, $d_{K} $ is strictly smooth (and Fr\'echet smooth)on $ T'(K)$  and the strict derivative coincides with the G\^ateaux derivative $($the Fr\'echet derivative$)$.
\end{Remark}


Now, we establish the following result under the following proximity condition on $K$:  $T'(K)$ is dense in $X \diagdown K$. That would serve as a  part of the next result of this paper.
\begin{Proposition}
\label{Prop-01}
Let $X$ be a smooth normed linear space. Let $K$  be a nonempty closed set with the set $T'(K)$ dense in $X \diagdown K$. Suppose $x \in X \diagdown K$ is such that for every $y_n \in T'(K)$ with $y_n \rightarrow x$ the sequence  $\left\{f_{y_n-p(y_n)}\right\}$ is  $w^*$-convergent. Then the distance function $d_K$ generated by $K$ is strictly smooth  at $x$.
\end {Proposition}
\begin{proof}
To prove that $d_K$ is strictly smooth at $x \in X \diagdown K$, given  Remark \ref{rmk1}, it suffices to show that $\partial d_{K}(x)$ is a singleton.\\
 Let $y_n \in T'(K)$ be any sequence such that $y_n \rightarrow x$. By definition of the upper limit, for each $n \in \mathbb N$ there exists $ z_n \in X \diagdown K$ and $t_n > 0$ such that
 $$\left\| z_n - y_n\right\|+t_n < \frac{1}{n}, {\rm and}$$
 \begin{eqnarray*}
 d_{K}^\circ(y_n; y)-\frac{1}{n} & \leqslant &   \dfrac{d_{K} (z_n+t_n y)- d_{K}(z_n)}{ t_n }  \\
 {\rm hence,~~}  \limsup_{n\rightarrow \infty} d_{K}^\circ(y_n; y) & \leqslant & \limsup_{{z\rightarrow x},~ {t\downarrow 0}} \frac{ d_{K}(z+ty)- d_{K}(z)}{t}\\
  & = & d_{K}^\circ(x; y).
 \end{eqnarray*}
 Since $y_n \in T'(K)$ with  $y_n \rightarrow x$, so by Lemma \ref{lm1},  $\partial d_{K}(y_n)= \partial \left\| y_n - P_{K} (y_n)\right\| $ is a singleton,  we have $ d_{K}^\circ(y_n; y)= d'_{K} (y_n)(y)= f_n (y),$ where $f_n (y_n - P_{K} (y_n))= \left\| y_n - P_{K} (y_n)\right\|$, that is $f_n = f_{y_n-p(y_n)}$, with $y_n \rightarrow x$ and $y_n \in T'(K)$, hence $w^{\ast}$-convergent. Let $f_{y_n-p(y_n)} \rightarrow f$ in $w^{\ast}$-topology, by $w^{\ast}$-upper semicontinuity of $\partial d_{K}$ (see, e.g., \cite[Proposition 7.3.8(c)]{Sch}), we must have  $f \in \partial d_{K}(x).$\\
 Thus, for all $y \in X$ with $\|y\|=1$ and for every sequence $y_n \in T'(K)$ with  $y_n \rightarrow x$, we have $\lim_{n\rightarrow \infty} d_{K}^\circ(y_n;y)$ exists in $\partial d_{K}(x)(y),$ so linear in $y$ and
\begin{eqnarray*}
\lim_{n\rightarrow \infty} d_{K}^\circ(y_n; y) & \leqslant &   d_{K}^\circ(x; y).
\end{eqnarray*}

Now, we prove that the reverse of the last inequality holds for some $y_n \in T'(K)$ with  $y_n \rightarrow x$. Which proves that $d_{K}^\circ(x; y)$ is linear in $y$, and it follows that $\partial d_{K}(x)$ is a singleton.\\

For, let  $y \in X$ with $\|y\|=1$, then by definition of  $d_{K}^\circ(x; y)$, given any $ \epsilon > 0 $  there exist $z_n \in X \diagdown K$ with $z_n \rightarrow x$  and $t_n \downarrow 0$ such that for each $n,$ we have
 \begin{eqnarray*}
 d_{K}^\circ(x; y)-\frac{\epsilon}{2} & \leqslant &   \dfrac{d_{K} (z_n+t_n y)- d_{K}(z_n)}{ t_n }.
  \end{eqnarray*}
  Since $T'(K)$ is dense in $X \setminus K$, choose $y_n \in T'(K)$ such that $\left\| z_n +t_ny -y_n\right\| < t_n^2.$ Then $d_{K} (z_n+t_n y) \leqslant d_{K} (y_n) + t_n^2$ and $d_{K} (z_n) \geqslant d_{K} (y_n-t_ny) - t_n^2.$ Thus for all sufficiently large $n$, we have
\begin{eqnarray*}
 d_{K}^\circ(x; y)-\frac{\epsilon}{2} & \leqslant &   \dfrac{d_{K} (y_n)- d_{K}(y_n-t_ny)}{ t_n } + 2t_n  \\
  & \leqslant & d_{K}^\circ(y_n; y)+ \frac{\epsilon}{2}+2t_n.
 \end{eqnarray*}
 We justify the last inequality. Given $ \epsilon > 0 $  there exists $ \delta >  0$ such that for each $n,$ we have
 \begin{eqnarray*}
 d_{K}^\circ(y_n; y)+\frac{\epsilon}{2} & \geqslant &  \sup _{0< t< \delta}\dfrac{d_{K} ((y_n-t y)+t y)- d_{K}(y_n-t y)}{ t}    \\
  & = &  \sup _{0< t< \delta} \dfrac{d_{K} (y_n)- d_{K}(y_n-ty)}{ t }.
 \end{eqnarray*}
 Since $t_n \downarrow 0,$ so for all sufficiently large $n, t_n < \delta$. Thus for all sufficiently large $n,$
\begin{eqnarray*}
 d_{K}^\circ(y_n; y)+\frac{\epsilon}{2} & \geqslant &  \dfrac{d_{K} (y_n)- d_{K}(y_n-t_ny)}{ t_n }.
 \end{eqnarray*}
 Hence, the inequality is justified.  Finally, we have
 \begin{eqnarray*}
d_{K}^\circ(x; y)-\epsilon & \leqslant & \lim_{n\rightarrow \infty} d_{K}^\circ(y_n; y),
\end{eqnarray*}
and the required inequality follows. This completes the proof of the proposition.
\end{proof}
\begin{Remark}
It should be noticed that the merely density condition is not sufficient for the conclusion of Proposition \ref {Prop-01}. Let us consider $K= \left\{ x \in \mathbb R^2| \left\|x\right\| =1\right\}$, the unit sphere in smooth space $X=\mathbb R^2$, then $T'(K)= X \diagdown (K \cup \left\{ 0 \right\})$. Put $x=0,$ then we can find sequences $\left\{ y_n\right\}$ of nonzero vectors such that the sequences converge to zero but $\left\{f_{y_n-p(y_n)}\right\}$ are not convergent, so $d_K$ is not strictly smooth at $x=0$. The subdifferential of $d_K$ is the closed unit ball and it is easy to see that $d_K$ is  not smooth at $x=0$.
\end {Remark}

Before proceeding further to establish our main results, note that if $X$ is smooth then for any  $z_n \in E(K)$ and every $ p(z_n) \in P_K(z_n)$, the subdifferential   $\partial \left\|z_n- p(z_n)\right\| $ is a singleton which depends on $z_n$ and $ p(z_n)$ both. Since $P_K(z_n)$ is set-valued, so $f_{z_n- p(z_n)} \in \partial \left\|z_n- p(z_n)\right\| $ need not be  unique corresponding to given $z_n \in E(K).$ Indeed, every sequence $\left\{z_n\right\}$ in $E(K)$ determines (possibly uncountably)  many sequences $ \left\{f_{z_n- p(z_n)}\right\}$. Hence,  when we say that $ \left\{f_{z_n- p(z_n)}\right\}$ is ($w^{\ast}$-or norm) convergent for every sequence $\left\{z_n\right\}$ in $E(K)$ with $ z_n\rightarrow x $, we mean that  for every $ p(z_n) \in P_K(z_n)$ the sequence $ \left\{f_{z_n- p(z_n)}\right\}$ where $ f_{z_n- p(z_n)} \in \partial \left\|z_n- p(z_n)\right\| $, obtained in this way, is ($w^{\ast}$-or norm) convergent. It may be noted that they need not converge to the same ($w^{\ast}$-or norm) limit. For more details, see  Borwein et al. \cite[Corollary 9]{Borwn}  and Giles \cite{Giles1}.\\

The following proposition provides a necessary and sufficient condition for G\^ateaux differentiability, and a sufficient condition for Fr\'echet differentiability of the distance function of $K$ if $T'(K)$ is dense in $X \diagdown K.$ We do not assume the uniform differentiability conditions as in Giles \cite[Proposition 2]{Giles1}. Though, in \cite{Giles2} the author has given a characterization for the smoothness (Fr\'echet smoothness) of distance function without assuming the density of $E(K)$ but under the uniform differentiability constraints. Hence, our results strengthen that of Giles \cite{Giles1, Giles2}.

\begin{Proposition}
\label{Prop-1} Let $X$ be a normed linear space, and  $K$  be a nonempty closed set with the set $T'(K)$ dense in $X \diagdown K$. Then the following hold.\\\\
\hspace*{0.5cm} {\rm(i)} If the norm is smooth then the distance function $d_K$ generated by $K$ is strictly smooth at $x \in X \diagdown K$ if and only $\left\{f_{z_n- p(z_n)}\right\}$is  $w^*$-convergent for every sequence $\left\{z_n\right\}$ in $E(K)$ with $ z_n \rightarrow x $ and a nearest point $ p(z_n) \in P_K(z_n)$.\\\\
 \hspace*{0.5cm} {\rm(ii)} If the norm is Fr\'echet smooth then the distance function $d_K$ generated by $K$ is strictly smooth and Fr\'echet smooth at $x \in X \diagdown K$ if  $\left\{f_{z_n- p(z_n)}\right\}$is  norm convergent for every sequence $\left\{z_n\right\}$ in $E(K)$ with $ z_n \rightarrow x $ and a nearest point $ p(z_n) \in P_K(z_n)$.
\end{Proposition}

\begin{proof}
  \textbf{(i)} we consider the case when the norm is smooth. Suppose that for every sequence  $\left\{z_n\right\}$ in $E(K)$ with $ z_n \rightarrow x $, the sequence $\left\{f_{z_n- p(z_n)}\right\}$ is $w^{\ast}$-convergent. Since $E(K)$ contains $T'(K)$, it follows from Proposition \ref{Prop-01} that $d_K$ is strictly smooth at $x.$


Conversely, suppose that $d_{K}$ is strictly smooth at $x$. Let $f_{z_n- p(z_n)} \in \partial \left\| z_n - p(z_n)\right\|$. Since the norm is smooth, by Lemma \ref{lm0}, it follows that $f_{z_n- p(z_n)} \in \partial d_{K}(z_n)$. Let $f$ be the $w^{\ast}$-cluster point of $f_{z_n- p(z_n)}$, by $w^{\ast}$-upper semicontinuity of $\partial d_{K}$, we have $f \in \partial d_{K}(x)$, but $d_{K}$ is strictly smooth at $x$, hence the sequence $\left\{f_{z_n - p(z_n)}\right\}$ is $w^{\ast}$-convergent to $d'_{K}(x).$\\

 \textbf{ (ii)} We \textbf{now} consider the case when the norm is Fr\'echet smooth.
Suppose that for every sequence $\left\{z_n\right\}$ in $E(K)$ with $ z_n \rightarrow x $, the sequence $f_{z_n - p(z_n)}$ is  norm  convergent (so $w^{\ast}$-convergent to $d'_{K}(x)$). Then $d_{K}$ is strictly smooth at $x$.  It remains to prove the Fr\'echet smoothness only.\\

 Since $d_K$ is smooth at $x$, for any $t_n \rightarrow 0$ and any $y \in X$ with $ \|y\| = 1,$ we have
\begin{eqnarray*}
   \lim_{n \rightarrow \infty} \dfrac{d_{K} (x+t_n y)- d_{K}(x)}{ t_n } & = &  d'_K(x) (y)=  \lim_{n \rightarrow \infty} f_{z_n - p(z_n)}(y).
\end{eqnarray*}
But $f_{z_n - p(z_n)}\rightarrow d'_K(x)$ in  norm, therefore the last limit is  uniform  over  $ \|y\| = 1.$ Consequently, the left hand side limit is uniform over  $ \|y\| = 1$ and same for any sequence $t_n \rightarrow 0$. Therefore,  the limit
\begin{eqnarray*}
   \lim_{t \rightarrow 0} \dfrac{d_{K} (x+ty)- d_{K}(x)}{t}  =  d'_K(x)
(y)
\end{eqnarray*}
exists uniformly over  $\|y\| = 1$, \rm see~ Rudin \cite[Theorem 4.2]{Rud}. Hence $d_{K}$ is Fr\'echet smooth at $x$.
\end{proof}
\begin{Remark}
\label{rmk2}
It is known that for a nonempty closed set $K$  in a  Fr\'echet smooth  space $X$ the distance function $d_{K}$ is Fr\'echet smooth on $T'(K)$, see, e.g., \cite[Theorem (C)]{Bala88}. However, in Proposition \ref{Prop-1} (ii), we have shown both the Fr\'echet smoothness and strict smoothness of $d_{K}$ at any point $x \in X \diagdown K$ but on some expense of restrictions.
\end {Remark}
\section{Almost Sun and a Proximinality}\label{sec:MResl}
 We show that the almost sun property is a key ingredient in a number of results in this section. We recall that the norm on $X$ is said to be locally uniformly convex (LUR) if for every $x \in X$ with $||x|| =1 $ and given any sequence $\left\{y_n\right\}$ in $X$  with $||y_n|| \leqslant 1 $ such that $\displaystyle \|\frac{y_n +x }{2}\|\rightarrow 1 \Rightarrow y_n \rightarrow x.$
\begin{Remark}
\label{rem_L}
In  {\rm \cite [Theorem 4]{Dutt1}}, it has been  shown for an almost proximinal set $K$ that the local uniform convexity {\rm(LUR)} of the norm  is a sufficient condition on $X$ for the set $T'(K)$ to be dense in $X \diagdown K$.
\end {Remark}

The following Theorem signifies that the uniform smoothness (uniform Fr\'echet smoothness) of the distance function on a dense set will result in the smoothness (Fr\'echet smoothness) on $X \diagdown K$, if the norm on $X$ is LUR and smooth (Fr\'echet smooth).

\begin{Theorem}
\label{thm01} Let $X$ be a normed linear space with LUR and  smooth {\rm(}Fr\'echet smooth{\rm )} norm.
Let $K$ be a nonempty closed and almost proximinal set, let some subset $D(K)$ of  $E(K)$ be dense in $X \diagdown K$, and let $d_{K}$ be uniformly
smooth {\rm(}uniformly Fr\'echet smooth{\rm )} on $D(K)$.
Then the distance function $d_{K}$ is strictly smooth {\rm(}and Fr\'echet smooth{\rm)}  on $X\diagdown K.$
\end{Theorem}
\begin{proof} Let  $\bar{x} \in X \diagdown K$ and $\bar{r} > 0$ be arbitrary chosen.  Then due to the denseness of $D(K)$ in $X \diagdown K$, the set $D(K) \cap B(\bar{x}, \bar{r}) $ is nonempty. Since $d_{K}$ is uniformly smooth  (uniformly Fr\'echet smooth)  on $D(K) \cap B(\bar{x}, \bar{r})$. Therefore, for given $\epsilon > 0 $ and $y \in X$
 with $ \left\|y\right\|  =1$, there exists a $\delta (\epsilon, y)>0 \left(\delta(\epsilon)>0\right)$ such that
\begin{eqnarray*}
\left|\dfrac{ d_{K} (x + t y)- d_{K}(x)}{t } - f_{x-p(x)}(y)\right| < \epsilon~~~ \displaystyle {\rm ~for ~all~} x \in D(K) \cap B(\bar{x}, \bar{r}), 0< |t| < \delta.
\end{eqnarray*}

So,  for $x,z \in D(K) \cap B(\bar{x}, \bar{r})$ and for any $y \in X$
 with $ \left\|y\right\| =1,$ we have
\begin{eqnarray*}
\left|f_{z-p(z)}(y) - f_{x-p(x)}(y)\right|  & \leqslant &
\left|\dfrac{ d_{K} (z + t y)- d_{K}(z)}{t } - f_{z-p(z)}(y)\right| \\
 & + & \left|\dfrac{ d_{K} (z + t y)- d_{K}(z)}{t } - \dfrac{ d_{K} (x + t y)- d_{K}(x)}{t }\right| \\      &+ & \left|\dfrac{ d_{K} (x + t y)- d_{K}(x)}{t } - f_{x-p(x)}(y)\right|\\
 & < & 2 \epsilon + \frac{4}{\delta }\left\|z - x\right\| ~~~ \displaystyle {\rm ~for ~all~}  \frac{\delta}{2}< |t| < \delta\\
 & < & 6 \epsilon ~~~ \displaystyle {\rm ~for ~all~} \left\|z - x\right\|< \epsilon \delta.
\end{eqnarray*}
That is, the mapping $x \mapsto f_{x-p(x)}(y) (x \mapsto f_{x-p(x)} )$ is uniformly continuous on $D(K) \cap B(\bar{x}, \bar{r})$
 Since $D(K)$  is dense in $X \diagdown K$, this mapping has a unique continuous
extension on $B(\bar{x}, \bar{r})$, see Rudin \cite[page 99, exc. 13]{Rud}. But this implies that for any $x \in B(\bar{x}, \bar{r})$ and sequence $\left\{z_n\right\}$ in $D(K) \cap B(\bar{x}, \bar{r})$ converging to $x$, the sequence  $\left\{f_{z_n-p(z_n)}\right\}$ is $w^*$-convergent (respectively, norm convergent). Since the norm on $X$ is LUR, it follows from Remark \ref{rem_L} that $T'(K)$ is dense in $X$. Hence, by Proposition \ref{Prop-1} the distance function  $d_{K}$ is strictly smooth (respectively, Fr\'echet smooth) at $x$.
\end{proof}
Let us denote by $ E_r (K)$ the set
$$\left\{ x_{\circ}-r \frac{x_{\circ}-p(x_{\circ})}{\left\| x_{\circ}-p(x_{\circ}) \right\|} : x_{\circ} \in E(K),  p(x_{\circ}) \in P_{K}(x_{\circ}) ~ {\rm and}~  \left\| x_{\circ}-p(x_{\circ}) \right\| = d_K (x_{\circ})> r \right\}.$$
\hspace*{0.5cm} In Theorem \ref{thm01}, if we take $D(K)= E_r(K)$ for some $r > 0 $, then Theorem \ref{thm01} concludes for a particular situation, which is of our interest.
\begin{Corollary}
\label{cor01} Let $X$ be a normed linear space with LUR and  smooth {\rm(}Fr\'echet smooth{\rm )} norm.
Let $K$ be a nonempty closed almost proximinal set, let for some $r > 0$ the set $E_r(K)$ be dense in $X \diagdown K$,  and let $d_{K}$ be uniformly
smooth {\rm(}uniformly Fr\'echet smooth{\rm )} on $E_r(K)$. Then
the distance function $d_{K}$ is strictly smooth {\rm(}and Fr\'echet smooth{\rm)}  on $X\diagdown K.$
\end{Corollary}
One of the main results of this paper is to investigate the conditions on a nonempty closed set $K$ such that $E_r (K)$ is dense in $X  \diagdown K.$  Simple  observation reveals that for an almost sun $K$ the set $ E_r (K)$ is dense in $X  \diagdown K$. Indeed, we have the following result.
\begin{Lemma}
\label{lem03}
Let $K$ be a closed almost sun in a normed linear space $X$. Then for every $r > 0,$ the set $ E_r (K)$ is dense in $X  \diagdown K.$
\end{Lemma}
\begin{proof} Let $ S(K)$ denote the set of solar points of
$K$ . Since $K$ is an almost sun, the  set $ S (K)$ is dense in $X \diagdown K$. Therefore, it suffices to prove that $ E_r (K)$ is dense in $S(K).$\\
Let $y \in S(K)$ be any point with $p(y) \in P_{K}(y)$ (a \emph{luminosity point}, see, e.g., \cite[p. 12]{Alim}). For all $0< \epsilon < r $ and $|t|< \epsilon$, if we choose $$x_{\circ} = y + (r-t) \frac{y -p(y)}{\left\| y -p(y) \right\|}. $$
Then $x_{\circ} $ is also in  $S(K)$ and by definition of a sun $p(y)\in P_{K}(y)$ is a nearest point for $x_{\circ} $. Indeed, $P_{K}(y) \subset P_{K}(x_{\circ})$  that is there is a point $p(x_{\circ}) \in P_{K}(x_{\circ})$ such that $p(y) = p(x_{\circ})$,  and
$d_K(x_{\circ})= \left\| x_{\circ}-p(x_{\circ}) \right\|= \left\| x_{\circ}-p(y) \right\|= \left\| y-p(y) \right\|-t+r > r$.\\
 Since $\displaystyle \frac{x_{\circ}-p(x_{\circ})}{\left\| x_{\circ}-p(x_{\circ}) \right\|}= \frac{y-p(y)}{\left\| y-p(y)\right\|}$,  we have
\begin{eqnarray*}
&& x = x_{\circ}- r \frac{x_{\circ}-p(x_{\circ})}{\left\| x_{\circ}-p(x_{\circ}) \right\|} \in E_r (K)\\
&&~~ = x_{\circ}- r \frac{y-p(y)}{\left\| y-p(y) \right\|}\\
\text{that is,}&& x = y- t \frac{y-p(y)}{\left\| y-p(y)\right\|}, \\
\text{consequently,}&& \left\| x-y \right\|= |t|< \epsilon.
\end{eqnarray*}
Thus, for $y \in S(K)$ and $0 < \epsilon < r $, the point $\displaystyle x= y-t \frac{y-p(y)}{\left\| y-p(y)\right\|}$ is in $ E_r (K)$ for all $|t|< \epsilon $ such that $\left\| x-y \right\| < \epsilon.$ This proves that $ E_r (K)$ is dense in $S(K)$ for all $r>0$.
\end{proof}

Using Lemma \ref{lem03} and Corollary \ref{cor01}, we can deduce the differentiability of $d_{K}$ on $X\diagdown K,$ if it is uniformly differentiable on some dense set $ E_r (K).$
\begin{Corollary}
\label{core-LUR}
Let $X$ be a normed linear space with LUR and  smooth {\rm(}Fr\'echet smooth{\rm )} norm,
and $K$ be a nonempty closed almost sun in $X$.  Suppose $d_{K}$ is uniformly smooth {\rm(}uniformly Fr\'echet smooth{\rm )} on the set $E_r(K)$  for some $r > 0$. Then the distance function $d_{K}$ is strictly smooth {\rm(}and Fr\'echet smooth{\rm)}  on $X\diagdown K.$
\end{Corollary}
It may be worthy to find an almost sun which is not a sun. Dunham\cite{Dnhm} constructed Chebyshev sets in ${\cal C}[0,1]$ which are not suns. Braess\cite{Bras} has given $N$-parameter Chebyshev sets which are not suns whereas Dunham's sets were of 1-parameter and 2-parameters. We refer \cite[p. 32]{Alim} for more such sets. For an overview, we consider the following from Mhaskar and Pai \cite[page 474]{Pai}, where $v(\alpha,t)$ is a particular choice of the function $F(a,\cdot)$ given in \cite[Theorem 1]{Dnhm}.
 \begin{Example}
{\rm  Following Dunham\cite{Dnhm}, consider the real Banach space $ X = {\cal C}[0,1]$ with sup norm.\\\\ Let $K =\{v(\alpha,t)| \alpha \geqslant 0\}$,
where  $v(\alpha, t)=\begin{cases}
\dfrac{2 + \alpha}{1 + t/\alpha} &\text {if $\alpha > 0 $}\\
0  &\text {if $\alpha=0$},
\end{cases}$
for $t$ in $[0,1]$}.
\end {Example}
Then $K$ is a Chebyshev set, see the details in \cite{Pai} or \cite[Theorem 1]{ Dnhm}. Since $0$ is an isolated point of $K$, hence  $K$ is not a sun.

It is worth mentioning here that the Chebyshev set $K$ is locally compact. Also, we show that every negative function is a solar point of  $K$ . Let $y(t) < 0$ for all  $t \in [0,1]$.
\begin{eqnarray*}
\text{Then, clearly}&&  p(y) = 0 ~ \text{on} ~[0,1]. \\
\text{Now, for}~  \lambda  \geqslant 0, && y_{\lambda} =y+ \lambda  (y - p(y) ) \\
\text{that is ,}&&  y_{\lambda} =  (1 + \lambda) y < 0.\\
\end{eqnarray*}
This shows that $ y_{\lambda}$ has the same nearest point $p(y) = 0$ in $K$ for every $\lambda>0.$ Hence, every negative
function is a solar point.\\
\hspace*{0.5cm} On the other hand, any function in $X  \diagdown K$ which is positive somewhere can not be a solar point. For illustration, suppose that $x \in X  \diagdown K$ is such that $\min x(t) > 1$, then $x$ has a unique nearest point $p(x)= v(\alpha, t), \alpha > 0$. It is easy to see that $x+ \lambda  (x - p(x) )$ can not have the $p(x)$ as a nearest point for all $\lambda > 0$.\\
\hspace*{0.5cm} Note that the set of negative functions $G_{\rm neg}$ is open in $X$. Of course, $G_{\rm neg}$ is not dense in $X\diagdown K$. Nevertheless, we have an exact larger set $G_{\rm neg}$ as the set of solar points of  $K$. Thus, the Chebyshev set of Dunham \cite{Dnhm} encourages the search for almost sun which is not
a sun. Unfortunately, we have been unable to construct an example of an almost sun which is not a sun. Such a search must be in non-Hilbert spaces.\\\\
\hspace*{0.5cm}We observe that the notion of almost sun is not merely a tidier form for Giles \cite[page 462]{Giles1}, but it also provides a non-trivial illustration of the  proximinal condition. Moreover, Lemma \ref{lem03} motivates to improve the results of Giles \cite[page 462]{Giles1}.
\begin{Theorem}
\label{thm1_open}
Let $X$ be a normed linear space with uniformly smooth (uniformly Fr\'echet smooth) norm, and $K$ be a closed almost sun in $X$. Then $d_K$ is  smooth (Fr\'echet smooth) on $X  \diagdown K.$
\end{Theorem}
\begin{proof}
In the Giles proof, it requires only the denseness of $ E_r (K)$ in $X  \diagdown K $ for some $r>0$,
 which follows from Lemma \ref{lem03} for the almost sun $K.$
\end{proof}
The above Theorem  \ref{thm1_open} enables us to give a better characterization for the convexity of a set. Which follows directly from  Corollary \ref{cor01} and Proposition \ref{prop_open}.
Before that, we need the following notion of $\delta$-sun introduced by Vlasov\cite{Vlas3}.\\\\
A nonempty closed set $K$ is called a $\delta$-sun if for every point $x \in X  \diagdown K $ there exists a sequence $z_n \rightarrow x, z_n \neq x$
such that
\begin{equation}
\label{delsun}
\lim_{n\rightarrow \infty}\frac{d_K(z_n)-d_K(x)}{\|z_n-x\|}=1
\end{equation}
It is easy to see that an almost sun is a $\delta$-sun. The following result relates to LUR and $\delta$-sun.
\begin{Theorem}
\label{thm-lur}
(\cite[Theorem 1.2 (d,e)]{Balag})Let $X$ be a Banach space, and $K$ be a nonempty closed subset of  $X$. If $X$ is LUR, then
 $\delta'S \subset T'(K)$, $\delta S \cap E(K)\subset T'(K)$ and $\delta' S \cap E(K) \subset T(K) \cap AC(K) = T'(K)$. \\
 (Here, $\delta S$ is the set of points $x \in X  \diagdown K $ where $K$ is a $\delta$-sun,  and $\delta' S$ is the set of points $x \in X  \diagdown K$ for which there exist $h$ with $\|h\|=1$, $t_n \downarrow 0$ such that for $z_n= x+t_n h$ we have  (\ref{delsun})).
\end{Theorem}
\begin{Remark}
Note that for the set of solar points $S(K)$, we have  $S(K) \subset \delta S$. For an almost sun $K$, $S(K)$is dense in $X  \diagdown K$ and $S(K) \subset  E(K)$, hence if $K$ is an almost sun and $X$ is LUR then  density of $T'(K)$ in $X  \diagdown K$ follows from Theorem \ref{thm-lur}. However, $T'(K)$ is known
to be dense in $X  \diagdown K$  under weaker constraints ( see \cite[Theorem 4]{Dutt1}).
\end{Remark}
\begin{Theorem}
\label{thm-c}
Let $X$ be a Banach space, and $K$ be a nonempty closed subset of  $X$. Then $K$ is convex if any one of the following conditions holds:
\begin{description}
\item {\rm(i)} $X$ has uniformly smooth norm and  $K$ is an almost sun.
\item {\rm(ii)} $X$ is LUR with rotund dual $X^\ast,$ the distance function $d_K$ is uniformly smooth on the set $E_r(K)$  for some $r > 0$ which is dense in $X  \diagdown K  $, and satisfies $\left\| d'_K(x) \right\|=1$ for all $x \in X \diagdown K.$
\item {\rm(iii)} $X$ has LUR and Fr\'echet smooth norm with rotund dual $X^\ast$, the distance function $d_K$ is uniformly Fr\'echet smooth on the set $E_r(K)$  for some $r > 0$ which is dense in $X  \diagdown K$.
\end{description}
\end{Theorem}
\begin{proof}
(i) M. Day \cite{Day} showed that $X$ is uniformly smooth if and only if its dual space  $X^\ast$  is uniformly convex. Hence, $X^\ast$ is rotund. It is easy to verify that an almost sun is a $\delta$-sun, so $K$ is a  $\delta$-sun. Invoking \cite[Theorem 4.2]{Alim}, the $\delta$-sun $K$ is convex.\\\\
(ii) Since the dual space $X^\ast$  is strictly convex, the norm on $X$ is smooth. Invoking Corollary \ref{cor01}, we see that $d_K$ is smooth on $X  \diagdown K$. Invoking Proposition \ref{prop_open}, the set $K$ is convex.\\
(iii) Similar to the proof as above in (ii).
\end{proof}

\begin{Remark}
The assertions $(ii)$ and $(iii)$ in the above Theorem have a great analogy to \cite[Theorem 1.4]{Balag}. However, our results  $(ii)$ and $(iii)$  can not be deduced using \cite[Theorem 1.4]{Balag}.
\end{Remark}

Observe  that if the distance function $d_K$ generated by a proximinal set $K$ is smooth on $X \diagdown K,$ then we have$\left\| d'_K(x) \right\|=1$ for all $x \in X \diagdown K.$
So if $X$ has uniformly smooth norm, then every  proximinal set $K$ is convex provided that  $E_r(K)$ is dense in $X\diagdown K$ for some $r > 0$.This implies that every proximinal and almost sun $K$
is convex.\\
	


The following result asserts that proximinality is equivalent to Chebyshev property for almost sun.
\begin{Theorem}
Let $X$  be a normed linear space with uniformly smooth and rotund norm. Then a closed almost sun $K$ is Chebyshev if and only if $K$ is proximinal.
\end{Theorem}
\begin{proof} By definition, a Chebyshev set is proximinal. Now, suppose that the set $K$ is proximinal, we show that $K$ is Chebyshev.
To see this,  suppose that  a point $x \in X \diagdown K$
has two nearest points $p_1(x), p_2(x) \in  K.$ If $x_1$ and $x_2$ denote the unit vectors in the direction of $x-p_1(x)$ and $x-p_2(x)$ respectively. Since the norm of  the space $X$  is uniformly smooth and rotund and $K$ is an almost sun,
using Theorem \ref{thm1_open}, it follows that  $d_K$ is smooth at $x$ and
  $d'_K(x) = f_{x_1} = f_{x_2}.$  Since $ f_{x_1} (\frac{x_1+ x_2}{2}) = \left( \frac{ f_{x_1} (x_1)+  f_{x_1}(x_2)}{2} \right)= 1$, which implies that $\| \frac{x_1+ x_2}{2} \|= 1$, a contradiction to the rotundity. This proves the uniqueness of the nearest point.
  \end{proof}
Since every Hilbert space is rotund and uniformly smooth, and every Chebyshev set is proximinal, we have a partial result regarding the convexity of Chebyshev sets in a Hilbert space.
\begin{Corollary}
\label{thm-ref}
In  a Hilbert space every Chebyshev set which is almost sun is convex.
\end{Corollary}
%
Thus the problem of convexity of Chebyshev set in a Hilbert space is equivalent to the existence of a Chebyshev set $K$ which is not a sun at every point of some open ball in $X \diagdown K.$\\\\
It is known that in any reflexive Banach space $X$ with Kadec norm (such
spaces are called Efimov-Stechkin spaces \cite{Alim}), every nonempty closed set K has a set $ E(K)$ dense in $X \diagdown K $, see  Lau \cite[page 794]{Lau78}.  In particular, every Hilbert space has this property.  If for some $r>0$ the set $E_r (K)$ is dense in $X  \diagdown K$ then $ E(K)$ is dense in $X  \diagdown K $. Form Theorem \ref{thm-c} and succeeding discussions, it follows that if the converse were true, in  a Hilbert space every Chebyshev set must be convex.\\


It is easy to verify that the Vlasov's differentiability condition is a consequence of the almost sun property and so we have the following  result which is more general than the above Corollary.\\
\begin{Theorem}
\label{thm-1}
 Let $X$ be a Banach space with rotund dual $X^{\ast}$ and let $K$ be  nonempty closed set in  $X$. If $K$ is an almost sun, then $K$ is convex.
\end{Theorem}
\begin{proof}
Since $K$ is almost sun, by Lemma \ref{lem03},  for every $r>0$ the set  $ E_r (K)$ is dense in $X  \diagdown K $. Therefore, the proof follows from the ending theorem of Giles\cite[page 463]{Giles1} .
\end{proof}
\begin{Remark}
Since an almost sun is $\delta$-sun, and in a Banach space a $\delta$-sun is convex if and only if $X^\ast$ is strictly convex, see, e.g., \cite[Theorem 4.2]{Alim}.
 Hence, the above result is a simple corollary of \cite[Theorem 4.2]{Alim}. But we have provided another proof for the result under a slightly stronger condition.
\end {Remark}

Since a convex proximinal set is a sun, we have the following  characterization and equivalent conditions for convexity of Chebyshev and proximinal sets respectively.
\begin{Corollary}
 Let $X$ be a Banach space with rotund dual $X^{\ast}$, and let $K$ be a Chebyshev set in  $X$. Then  $K$ is convex if and only if it is an almost sun.
\end{Corollary}


\begin{Remark} It is worth mentioning that in any smooth Banach space a proximinal set is convex if and only if it is a sun, see e.g. \cite[Theorem 4.1]{Alim}.\\
  From Theorem \ref{thm-1} it follows that in any  Banach space $X$ with rotund dual $X^{\ast}$,  a proximinal set  $K$ in $X$ is convex if and only if it is an almost sun.  It is also worth mentioning that this result partially extends \cite[Corollary 3]{Bala88}.
\end {Remark}


%
%
%
%
%
%
%

\textbf{Acknowledgment:}
 The author is indebted to the anonymous referee for his meticulous reading of the initial manuscript, several valuable and helpful suggestions and comments with appropriate references, which significantly improved the quality of the paper.\\\\
\textbf{Compliance with ethical standards}\\
\textbf{Ethical approval}\\The work included in this research paper comprises the results of the independent and original investigation carried out. The material obtained (and used) from other resources has been duly acknowledged in the reference.\\\\
\textbf{Funding:} This study was funded by UGC-BSR Start-Up Grant No. F.30.12/2014(BSR), UGC, Govt. of India, New Delhi-110012.\\\\
\textbf{Ethical approval}\\
This article does not contain any studies with human participants performed by any of the authors.

Triloki Nath\\
Department of Mathematics and Statistics, School of Mathematical and Physical Sciences, Dr. Harisingh Gour Vishwavidyalaya, Sagar, Madhya Pradesh-470003, India.\\
\emph{E-mail address}: tnverma07@gmail.com
\end{document}